\documentclass[12pt,reqno]{amsart}
\pdfoutput=1
\usepackage[headings]{fullpage}
\usepackage{amssymb,amsmath,bbm,tikz}
\usepackage{graphicx,color,enumerate, float}
\usepackage[all,cmtip]{xy}
\usepackage{url}
\usepackage[margin=1.2in]{geometry}
\usepackage{verbatim}
\usepackage{pb-diagram}

\usepackage[bookmarks=true,%
    colorlinks=true,%
    linkcolor=blue,%
    citecolor=blue,%
    filecolor=blue,%
    menucolor=blue,%
    urlcolor=blue,%
    breaklinks=true]{hyperref}

\newtheorem{theorem}{Theorem}[section]
\theoremstyle{definition}

\newtheorem{lemma}[theorem]{Lemma}

\newtheorem{corollary}[theorem]{Corollary}

\newtheorem{example}[theorem]{Example}

\renewcommand{\setminus}{{\smallsetminus}}

\def\be{\begin{equation}}
\def\ee{\end{equation}}

\graphicspath{{figures/}}

\begin{document}

\title[Generators of the Steinberg Module]{Explicit generators of the Steinberg Module of the mapping class group} 

\author{Ingrid Irmer}
\address{SUSTech International Center for Mathematics\\
Southern University of Science and Technology\\Shenzhen, China
}
\address{Department of Mathematics\\
Southern University of Science and Technology\\Shenzhen, China
}
\email{ingridmary@sustech.edu.cn}

\thanks{
 {\em Mathematics Classification.} Primary 57K20. Secondary 57R70.
\newline
{\em Key words and phrases: Mapping class groups, Moduli space, Morse Theory.
}
}


\begin{abstract}A conjecture of Broaddus is proven, giving a simple characterisation of a representative of the unique orbit of the action of the mapping class group on the homology of Harvey's complex of curves for any genus surface. As an application, the kernel of the action of the mapping class group of a genus $g$ surface on the Steinberg module is shown to be trivial.
\end{abstract}

\maketitle

{\footnotesize
\tableofcontents
}


\section{Introduction}
\label{sec.intro}

It was shown in \cite{Harer}, Theorem 3.5, that the complex of curves, $\mathcal{C}_{g}$, of a closed, orientable topological surface $\mathcal{S}_{g}$ of genus $g\geq 2$ has the homotopy type of an infinite wedge of spheres of dimension $2g-2$. In \cite{Broaddus}, an algorithm for constructing a homotopically nontrivial sphere was given, and an explicit example was computed in genus 2. This paper proves a conjecture showing that a considerably simpler construction of spheres: ``Broaddus spheres'' - of which there is one in each genus - could be used. \\

The reduced homology group $\tilde{H}_{2g-2}\bigl(\mathcal{C}_{g};\mathbb{Z}\bigr)$ is also referred to as the Steinberg module $\mathrm{St}(\mathcal{S}_{g})$. Since there is a simplicial action of the mapping class group $\Gamma_{g}$ of $\mathcal{S}_{g}$ on $\mathcal{C}_{g}$, $\mathrm{St}(\mathcal{S}_{g})$ has the structure of a $\Gamma_{g}$-module. It was also shown in \cite{Broaddus} that $\mathrm{St}(\mathcal{S}_{g})$ is cyclic modulo the action of the mapping class group.\\

\begin{theorem}
\label{propaganda}
$\mathrm{St}(\mathcal{S}_{g})$ is generated by the mapping class group orbit of the Broaddus sphere of genus $g$.
\end{theorem}

A Poincar\'e duality group is a group whose homology and cohomology satisfy a duality property analogous to the Poincar\'e duality between the homology and cohomology of a compact manifold. This definition was generalised in \cite{BE} to the notion of a duality group. Informally, a duality group is a group whose homology and cohomology satisfy a property analogous to Poincar\'e duality, with a nontrivial ``dualising module'' taking the place of the orientation module. It was shown in Theorem 6.6 of \cite{Ivanov} that the mapping class group of a closed orientable surface is a virtual duality group, and that the dualising module is  $\mathrm{St}(\mathcal{S}_{g})$.\\

For genus $g$ at least 3, the center $Z_{g}$ of $\Gamma_{g}$ is trivial, as shown in Theorem 3.10 of \cite{FandM}. In Section 3.4 of \cite{FandM} it was shown that the center of $Z_{2}$ is generated by involutions that act trivially on $\mathcal{C}_{g}$. The  techniques of this paper can be used to relate the stabiliser subgroup of the action of $\Gamma_{g}$ on a generator of $\mathrm{St}(\mathcal{S}_{g})$.

\begin{corollary}
\label{kernel}
The kernel of the action of $\Gamma_{g}/Z_{g}$ on $\mathrm{St}(\mathcal{S}_{g})$ is trivial.
\end{corollary}

The systole function $f_{sys}:\mathcal{T}_{g}\rightarrow \mathbb{R}_{+}$ is the piecewise smooth map whose value at any point of $\mathcal{T}_{g}$ is the length of the systoles. The systole function is known to be a topological Morse function, \cite{Akrout}. These are defined in \cite{Morse}, and can be used to construct cell decompositions of topological spaces analogous to those of (smooth) Morse functions.\\

The central idea is that the Broaddus sphere can be described using the set of ``admissible boundary points'' of a set of minima, as defined in \cite{SchmutzMorse}. Informally, the sets of minima in this paper behave like ``unstable manifolds'' of critical points of $f_{sys}$. A set of minima, or a critical point, is labelled by a set of curves. The mapping class group maps sets of minima to sets of minima, and critical points to critical points. As this action is determined by the action of the mapping class group on curves, the action of the mapping class group on the Broaddus spheres is identical to the action on the sets of minima that they bound in $\mathcal{T}_{g}$. The Broaddus spheres are shown to represent nontrivial homology classes in $\mathrm{St}(\mathcal{S}_{g})$ by studying the way the fixed point sets of the action of the $\Gamma_{g}$ on the sets of minima restrict homotopies.\\

\textbf{Outline of the paper.} Section \ref{defns} first introduces some background and notations. For the benefit of readers who are not be familiar with Schmutz's sets of minima, a survey of their properties is given. The construction from \cite{Broaddus} of a generator of $\mathrm{St}(\mathcal{S}_{g})$, as well as a related example of a critical point due to Schmutz are given in \ref{Broaddussec}. The relationship between these ingredients is made clear in Section \ref{homotopysection}. The ``unstable manifold'' of the critical point in the form of a certain set of minima, is identified with the ball in Teichm\"uller space bounded by the sphere corresponding to Broaddus's conjectured generator of $\mathrm{St}(\mathcal{S}_{g})$. The vertices of this triangulated sphere are shown to correspond to particular boundary points of the set of minima. In Section \ref{finalproof} symmetries are used to show that the homology class of the sphere is nontrivial.


\section{Assumptions and Background}
\label{defns}

The orientable, closed, compact, connected topological surface of genus $g$ will be denoted by $\mathcal{S}_{g}$; when this surface is given a marked hyperbolic structure corresponding to a point in Teichm\"uller space, it will be denoted by $S_{g}$. \\

The Teichm\"uller space of $\mathcal{S}_{g}$ will be denoted by $\mathcal{T}_{g}$, the mapping class group by $\Gamma_{g}$, moduli space by $\mathcal{M}_{g}$, and Harvey's curve complex by $\mathcal{C}_{g}$. For convenience, the simplicial complex $\mathcal{C}_{g}$ will be confused with its geometric realisation.\\

A curve is a homotopy class of embeddings of $S^{1}$ into $\mathcal{S}_{g}$. When there is no possibility of misunderstanding, the word curve will also refer to a particular representative of the homotopy class, such as a geodesic. A set of curves is said to fill $\mathcal{S}_{g}$ if the complement of a set of representatives in minimal positition cuts the surface into polygons.\\

A curve $c$ determines an analytic map $L(c):\mathcal{T}_{g}\rightarrow \mathbb{R}_{+}$ whose value at any point is given by the length of the geodesic representative of $c$. A length function $L(A, C)$ is a linear combination of such functions with positive coefficients where $C$ is a finite set of $|C|$ curves, and $A=(a_{1}, \ldots, a_{|C|})\in \mathbb{R}^{|C|}_{+}$ the coefficients. Length functions satisfy many convexity properties, for example they are strictly convex along Weil-Petersson geodesics, \cite{Wolpert}. The Weil-Petersson metric will be assumed whenever a metric is needed on $\mathcal{T}_{g}$. \\

The systoles of $S_{g}$ is the set of shortest curves; this is always finite. The set of points in $\mathcal{T}_{g}$ at which the set of systoles is exactly $C$ will be denoted by $\mathrm{Sys}(C)$. The Thurston spine is the CW complex contained in $\mathcal{T}_{g}$ consisting of the set of points at which the systoles fill $\mathcal{S}_{g}$.\\


Teichm\"uller space has a decomposition into the thick part, where $f_{sys}$ is less than or equal to the Margulis constant, and the thin part, which is the complement of the thick part. For convenience, a decomposition into a $\delta$-thick part and $\delta$-thin part will be used, where $\delta>0$ might be less than the Margulis constant. The $\delta$-thick part of $\mathcal{T}_{g}$ is the subset of $\mathcal{T}_{g}$ on which $f_{sys}$ is greater than or equal to $\delta$, and will be denoted by $\mathcal{T}_{g}^{\delta}$.\\

It was shown in \cite{Ivanov} that $\mathcal{C}_{g}$ is $\Gamma_{g}$-equivariantly homotopy equivalent to the boundary of the thick part of Teichm\"uller space. The idea behind this is very simple. In the thin part of $\mathcal{T}_{g}$, the systoles are pairwise disjoint; at a given point they determine a multicurve $m$. The intersection of $\mathrm{Sys}(m)$ with the boundary of the thick part of $\mathcal{T}_{g}$ corresponds to the cell labelled by $m$ in the dual of $\mathcal{C}_{g}$. The same argument as in \cite{Ivanov} gives an embedding of $\mathcal{C}_{g}$ in the boundary of $\mathcal{T}_{g}^{\delta}$.  \\

\textbf{Sets of minima.} Let $C$ be a finite set of filling curves. From Lemma 1 of \cite{SchmutzMorse}, it is known that every length function given by a strictly positive linear combination of the lengths of curves in $C$ has a unique minimum at some point of $\mathcal{T}_{g}$. The set $\mathrm{Min}(C)$ defined in \cite{SchmutzMorse} consists of all points in $\mathcal{T}_{g}$ at which a length function of the form $L(A, C)$ for some $A\in \mathbb{R}_{+}^{|C|}$ has its minimum. A reference for sets of minima is Section 2 of \cite{SchmutzMorse}.\\

There is a surjective map $\phi_{C}:\mathbb{R}^{|C|}_{+}\twoheadrightarrow \mathrm{Min}(C)$, where here $\phi(A)$ is the point in $\mathrm{Min}(C)$ at which $L(A,C)$ has its minimum. Note that $\phi_{C}$ is not injective. This paper will be concerned with specific examples of sets $\mathrm{Min}(C)$ that are known to be differentiable cells with empty boundary.\\

Suppose $C=\{c_{1}, \ldots, c_{k}\}$ is a set of filling curves. The lengths of curves in $C$ are said to parameterise $\mathrm{Min}(C)$. For any $x_{1}$, $x_{2}$ in $\mathrm{Min}(C)$, the map $F:\mathrm{Min}(C)\rightarrow \mathbb{R}_{+}^{k}$ given by $x\mapsto (L(c_{1})(x), \ldots, L(c_{k})(x))$ has the property that $F(x_{1})=F(x_{2})$ only if $x_{1}=x_{2}$. When the rank of the Jacobian of $F(C)$ is constant on $\mathrm{Min}(C)$, as will always be the case in this paper, Section 2 of \cite{SchmutzMorse} showed how to use the lengths of the curves in $C$ to obtain a set of coordinates on $\mathrm{Min}(C)$. \\




\begin{lemma}[Lemma 4 of \cite{Bers}, or Proposition 1 from \cite{Thurston}]
\label{Thurstonprop}
Let $C'$ be any collection of curves on a surface that do not fill. Then at any point of $\mathcal{T}_{g}$, there are tangent vectors corresponding to directions in which the lengths of all the geodesics representing curves in $C'$ are increasing. 
\end{lemma}

It follows from Lemma 14 of \cite{SchmutzMorse} that points on $\partial\mathrm{Min}(C)$ in the interior of $\mathcal{T}_{g}$ are minima of strictly positive linear combinations of filling subsets of $C$. For a length function of the form $L(A, C')$, where $C'$ is a nonfilling subset of $C$, Lemma \ref{Thurstonprop} implies that any infima are realised as limits of sequences of points in $\mathrm{Min}(C)$. These are Schmutz's ``admissible boundary points'', \cite{SchmutzMorse}, Section 2. 



\section{From critical points to spheres}
\label{Broaddussec}


This section begins by surveying an important family of examples of critical points of the topological Morse function $f_{sys}$. These examples will later be related to the construction of $2g-2$-dimensional spheres in $\mathcal{C}_{g}$, whose construction is also given in this section.\\

\begin{example}[A critical point of $f_{sys}$ of index $2g-1$, from Theorem 36 of \cite{SchmutzMorse}]
\label{Schmutzexample}
Take a regular, right angled hyperbolic polygon $T$ with $2g+2$ sides. Four copies of $T$ can be glued together along four pairs of edges with a common vertex. This is done in such a way that each copy of $T$ shares one edge with two others. A right angled hyperbolic polygon with $8g-4$ sides is obtained. The construction with $g=2$ is shown in Figure \ref{genus2example}. This gives a fundamental domain of a surface. The sides of the fundamental domain are glued together in such a way that each edge of a copy of $T$ lies along a geodesic loop on the surface obtained by traversing exactly two edges, each on different copies of $T$. It transpires that the gluing maps are completely determined by this condition. When given a marking, this surface corresponds to a point in $\mathcal{T}_{g}$ that will be denoted by $p$.\\

It was shown in \cite{SchmutzMorse} that $p$ is a critical point of index $2g-1$ and that the set $C$ of systoles at $p$ is the set of curves that can be embedded in the graph on $S_{g}$ consisting of the edges of the four copies of $T$. This gives $2g+2$ systoles, each of which intersects exactly two other systoles. Let $c_{1}$ and $c_{2}$ be a pair of systoles that intersect. It was shown that $\mathrm{Min}(C\setminus \{c_{1}, c_{2}\})= \mathrm{Min}(C)$. For every $x\in \mathrm{Min}(C)$, the gradients $\{\nabla L(c)\ |\ c\in C\}$ lie in a $2g-1$ dimensional subspace of $T_{x}\mathcal{T}_{g}$. As $C\setminus\{c_{1}, c_{2}\}$ has no proper filling subsets and the rank of the Jacobian of the map $\mathcal{T}_{g}\rightarrow \mathbb{R}^{2g+2}$ given by $x\mapsto \bigl(L(c_{1}), \ldots, L(c_{2g+2})\bigr)$ is constant and equal to $2g-1$ everywhere on $\mathrm{Min}(C)$, it follows from Corollary 13 and Lemma 14 of \cite{SchmutzMorse} that $\mathrm{Min}(C)$ is a continously differentiable cell with empty boundary. \\

The cell $\mathrm{Min}(C)$ behaves like the unstable manifold of the critical point $p$. The definition of $\mathrm{Min}(C)$ implies that any vector in the tangent space to $\mathrm{Min}(C)$ at $p$ gives a direction in which the length of at least one curve in $C$ is decreasing. On a neighbourhood of $p$, the systoles are contained in $C$, so $p$ is a local maximum for $f_{sys}$ in $\mathrm{Min}(C)$.\\

The point $p$ is an element of a $\Gamma_{g}$-orbit of critical points. The systole function is invariant under the action of $\Gamma_{g}$. A $\gamma\in \Gamma_{g}$ maps a critical point with set of systoles $C$ to a critical point with systoles given by the image of $C$ under $\gamma$. Similarly, $\Gamma_{g}$ acts on sets of minima, mapping $\mathrm{Min}(C)$ to the set of minima of the image of $C$ under $\gamma$.\\


\end{example}

\begin{figure}[!thpb]
\centering
\includegraphics[width=0.6\textwidth]{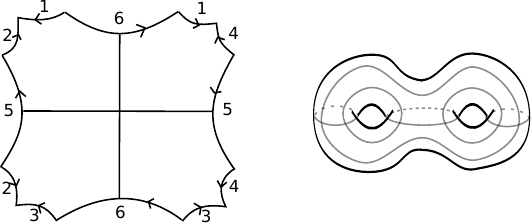}
\caption{The left side of this figure (not drawn to scale!) shows a fundamental domain of the genus 2 surface in Example \ref{Schmutzexample}. The systoles lie along the boundary and the edges shown. The numbers on the edges are intended to indicate the gluing maps. The right side of the figure shows the systoles on the surface. }
\label{genus2example}
\end{figure}

\begin{figure}
\centering
\includegraphics[width=0.8\textwidth]{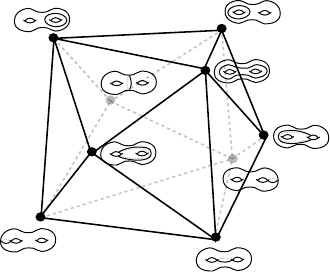}
\caption{This figure is reproduced from \cite{Broaddus} and shows the 1-skeleton of the triangulated sphere in $\mathcal{C}_{2}$ constructed by the algorithm below. Theorem \ref{maintheorem} shows this is homotopic in $\mathcal{C}_{g}$ to the intersection of $\mathrm{Min}(C)$ from Example \ref{Schmutzexample} with the embedding of $\mathcal{C}_{g}$ in $\partial\mathcal{T}_{g}^{\delta}$.  }
\label{genus2sphere}
\end{figure}

\textbf{Construction of the Sphere.} The $2g-2$ dimensional sphere in $\mathcal{C}_{g}$ from Conjecture 4.9 of \cite{Broaddus} is constructed as follows: Let $C$ be the set of $2g+2$ curves from Example \ref{Schmutzexample} and let $K_{2g+2}$ be a $(2g+2)$-gon. Choose a bijection between the vertices of $K_{2g+2}$ and the elements of $C$ with the property that adjacent vertices of $K_{2g+2}$ are labelled by intersecting curves. The symbol $\Phi_{g}$ will be used to denote the symplicial complex whose vertices correspond to diagonals of $K_{2g+2}$ and whose simplices correspond to sets of disjoint diagonals of $K_{2g+2}$.\\

An associahedron $\mathcal{K}_{n}$ is a polytope of dimension $n$ whose vertices correspond to all the distinct ways of parenthesising the word $x_{0}x_{1}\ldots x_{n+1}$. As shown in \cite{Lee2}, Theorem 1, the simplicial complex $\Phi_{g}$ is the boundary of the dual of the $2g-1$-dimensional associahedron $\mathcal{K}_{2g-1}$, and is therefore a triangulation of the $2g-2$-dimensional sphere. \\

The barycentric subdivision of $\mathcal{C}_{g}$ will be denoted by $\mathcal{C}_{g}^{\circ}$. Recall that each vertex $v$ of $\Phi_{g}$ is labelled by pair of nonintersecting curves $c(v)$ in $C$. A simplicial map 
\begin{equation}
\label{q}
q^{\circ}_{g}:\Phi_{g}\rightarrow \mathcal{C}_{g}^{\circ}
\end{equation}
 is contructed, where $q^{\circ}_{g}$ maps a vertex $v$ to the vertex in $\mathcal{C}_{g}^{\circ}$ labelled by the multicurve $m\bigl(C\setminus c(v)\bigr)$ consisting of the set of nontrivial homotopy classes of curves on the boundary of the surface obtained by cutting $S_{g}$ along the geodesic representatives of the curves $C\setminus c(v)$. Informally, $m(C')$ is contained in the boundary of the subsurface filled by $C'$.\\

The map $q^{\circ}_{g}$ can be seen to map simplices to simplices. The condition that the simplices of $\Phi_{g}$ correspond to sets of \textit{disjoint} diagonals of $K_{2g+2}$ implies that the image of a simplex under $\Phi_{g}$ is a simplex in $\mathcal{C}_{g}^{\circ}$.\\

\textbf{Defining the map} $\mathbf{q_{g}:\Phi_{g}\rightarrow \mathcal{C}_{g}}$. For the purposes of this paper, it will be simplest to work with the image in $C_{g}^{\circ}$ of $\Phi_{g}$ under $q^{\circ}_{g}$. This is a matter of taste; if one wants to obtain a sphere in $\mathcal{C}_{g}$, for any vertex $v$ with the property that $m\bigl(C\setminus c(v)\bigr)$ has more than one connected component, choose a single curve $m_{1}\bigl(C\setminus c(v)\bigr)$ in $m\bigl(C\setminus c(v)\bigr)$. It was shown in \cite{Broaddus} that $q^{\circ}_{g}$ is homotopic to a map from $S^{2g-2}$ to a subcomplex with the vertex in the image of $q^{\circ}_{g}$ labelled by $m\bigl(C\setminus c(v)\bigr)$ replaced by the vertex labelled by $m_{1}\bigl(C\setminus c(v)\bigr)$. This homotopy pushes simplices in the interior of a simplex of $\mathcal{C}_{g}^{\circ}$ onto the boundary of the simplex, as illustrated in Figure \ref{thehomotopy}. The map $q^{\circ}_{g}$, composed with one such choice of homotopy for every vertex not in $\mathcal{C}_{g}$, gives the map $q_{g}$. \\

For each $g\geq 2$ there is a map $q_{g}$, the homotopy class of which will be referred to as a Broaddus sphere.

\begin{figure}[!thpb]
\centering
\includegraphics[width=0.6\textwidth]{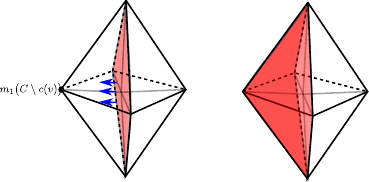}
\caption{Lower dimensional analogue of homotoping a subcomplex of $\mathcal{C}_{g}^{\circ}$ onto a subcomplex of $\mathcal{C}_{g}$. Edges and vertices of $\mathcal{C}_{g}^{\circ}$ not in $\mathcal{C}_{g}$ are shown in grey. }
\label{thehomotopy}
\end{figure}



\section{The Set of Minima}
\label{homotopysection}

The purpose of this section is to show that the Broaddus sphere in genus $g$ can be embedded in $\partial \mathcal{T}_{g}^{\delta}$ in such a way that it makes sense to think of it as the boundary of $\mathrm{Min}(C)$, where $C$ is the set of curves from Example \ref{Schmutzexample}.\\

\textbf{A thickening of the embedded curve complex. }
The dimension $3g-4$ of $\mathcal{C}_{g}$ is smaller than the dimension of $\mathcal{T}_{g}$, namely $6g-6$. In order to use the embedding of $\mathcal{C}_{g}$ into $\partial \mathcal{T}_{g}^{\delta}$ from \cite{Ivanov} it will be convenient to define a regular neighbourhood of the image of the embedding. An abuse of notation will be used to simplify the exposition; the image of the embedding of $\mathcal{C}_{g}$ or $\mathcal{C}_{g}^{\circ}$ in $\partial \mathcal{T}_{g}^{\delta}$ will also be referred to as $\mathcal{C}_{g}$ or $\mathcal{C}_{g}^{\circ}$.\\



Ivanov's embedding of $\mathcal{C}_{g}$ in $\partial \mathcal{T}_{g}^{\delta}$ has the property that a vertex labelled by a curve $c$ is mapped into $\mathrm{Sys}(\{c\})$, an edge labelled by $\{c_{1}, c_{2}\}$ is mapped into the union of $\mathrm{Sys}(\{c_{1}\})$, $\mathrm{Sys}(\{c_{2}\})$ and $\mathrm{Sys}(\{c_{1}\cup c_{2}\})$, etc. \\

For each multicurve $m$ labelling a top dimensional simplex of $\mathcal{C}_{g}$, the intersection of $\mathrm{Sys}(m)$ with $\partial \mathcal{T}_{g}^{\delta}$ is an embedded submanifold without boundary. For a  multicurve $m'$ with fewer connected components, the boundary of $\mathrm{Sys}(m')\cap \partial \mathcal{T}_{g}^{\delta}$ lies along $\{\mathrm{Sys}(m'')\cap \partial \mathcal{T}_{g}^{\delta}\ | \ m'\subsetneq m''\}$. As these extend to embedded submanifolds with boundary in a neighbourhood of $\partial\mathcal{T}_{g}^{\delta}$, the same arguments as in \cite{Ivanov} show that it is possible to embed a regular neighbourhood $\pi:\mathcal{N}_{g}\rightarrow\mathcal{C}_{g}^{\circ}$ into the thin part of $\mathcal{T}_{g}$ in such a way that the duality between the sets of systoles and the cells of the regular neighbourhood is preserved.  \\

\textbf{Construction of the map }$\mathbf{q_{C,g}}$. Recall the definition of $q^{\circ}_{g}:\Phi_{g}\rightarrow \mathcal{C}_{g}^{\circ}$ from the previous section, where $\Phi_{g}$ is a triangulation of a $2g-2$ dimensional sphere. Another map $q_{C,g}$ from $\Phi_{g}$ to $\mathcal{C}_{g}^{\circ}$ will now be constructed. This map factors through a map homotopic to the inclusion map into $\partial\mathcal{T}_{g}^{\delta}$ of a $2g-2$-dimensional sphere in the intersection $\partial\mathcal{T}_{g}^{\delta}\cap \mathrm{Min}(C)$.\\

The next construction is motivated by the discussion of admissible boundary points from Section \ref{defns}. Let $c_{1}, c_{2}$ be a pair of disjoint curves in $C$, and $C':=C\setminus \{c_{1}, c_{2}\}$ a nonfilling set of curves. The curves $ \{c_{1}, c_{2}\}$ correspond to a diagonal of $K_{2g+2}$. Then the length function $L(C')$ obtained by taking the sum of the lengths of curves in $C'$ has no minimum in $\mathcal{T}_{g}$ by Lemma \ref{Thurstonprop}. It can be thought of as having its infimum at the limit of a sequence of points whose parameters $(a_{1}, \ldots, a_{2g+2})\in \mathbb{R}^{2g+2}_{+}$ approach $\frac{1}{2g}(0, 0, 1, 1, \ldots, 1)$.\\

For convenience, from now on the normalisation factors on the parameters in $\mathbb{R}_{+}^{2g+2}$ will be dropped. For the arguments given here, it is only of interest which of the parameters are close to zero.\\

For small $\delta$, the length function $L(C')$ nearly reaches its infimum on $\mathrm{Min}(C)\cap\partial\mathcal{T}_{g}^{\delta}$. Denote by $v(C')$ a point at which $L(C')$ is minimised on $\mathrm{Min}(C)\cap\partial\mathcal{T}_{g}^{\delta}$. Similarly, for any other vertex of $\Phi_{g}$, define a corresponding point in $\mathrm{Min}(C)\cap\partial\mathcal{T}_{g}^{\delta}$. The only restriction imposed here is that any choices are made in such a way that the embedding is invariant under the action of the subgroup of $\Gamma_{g}$ that maps $\mathrm{Min}(C)$ to itself. \\

Imposing the restriction that any choices in the construction of the embedding are made in a $\Gamma_{g}$-equivariant way is not a problem, for reasons that will now be explained. The subgroup of $\Gamma_{g}$ that maps $\mathrm{Min}(C)$ to itself permutes the curves in $C$ in a way that preserves or reverses the cyclic ordering. This subgroup could only map a minimum of $L(C')$ on $\mathrm{Min}(C)\cap\partial\mathcal{T}_{g}^{\delta}$ to a different minimum if it maps $C'$ to itself. This is only possible if $c_{1}$ and $c_{2}$ represent a pair of opposite vertices of $K_{2g+2}$. In this case, $m(C')$ consists of a single separating curve that is invariant under the action of the subgroup. It follows from Riera's theorem that the minima of $L(C')$ on $\mathrm{Min}(C)\cap\partial\mathcal{T}_{g}^{\delta}$ occur where $m(C')$ is a systole. By symmetry, such points are therefore invariant under the action of the subgroup of $\Gamma_{g}$ that maps $\mathrm{Min}(C)$ to itself. It follows that any choices can be made in a $\Gamma_{g}$-equivariant way.\\

Every edge of $\Phi_{g}$ determines a pair of endpoints $v(C')$ and $v(C'')$ with $m(C')$ and $m(C'')$ disjoint multicurves. The midpoint of this edge is represented by a tuple in $\mathbb{R}^{2g+2}_{+}$ with parameters close to zero in 3 or 4 positions, and all other parameters close to 1. The parameters that are small consist of the union of the parameters that are small at the vertices. \\

To be more precise: $C'=C\setminus\{c_{i}\cup c_{k}\}$ and $C''=C\setminus\{c_{l}, c_{m}\}$ where $c_{l}$ or $c_{m}$ (but not both) might be in $\{c_{i}, c_{j}\}$. Then the midpoint of the edge connecting $v(C')$ and $v(C'')$ has parameters  $a_{i}, a_{j}, a_{l}$ and $a_{m}$ close to zero, and all others close to 1. The midpoint of this edge can be thought of as a minimum of the length function $L(C\setminus \{c_{i}, c_{j}, c_{l}, c_{m}\})$ on $\partial\mathcal{T}_{g}^{\delta}$. Between the midpoint and each of the endpoints, the edge is defined by taking a linear interpolation of the parameters. Maps of higher dimensional simplices of $\Phi_{g}$ into $\mathrm{Min}(C)$ are defined analogously, with midpoints contained in $\partial\mathcal{T}_{g}^{\delta}\cap \mathrm{Min}(C)$.\\

The image of the simplices of $\Phi_{g}$ are not guaranteed to be contained in $\partial \mathcal{T}_{g}^{\delta}$, but for sufficiently small $\delta$, $\mathcal{N}_{g}$ can be chosen in such a way that they are contained in $\mathcal{N}_{g}$. Since $\Phi_{g}$ is a $2g-2$-sphere, and $\mathrm{Min}(C)$ a $2g-1$-ball, inside $\mathcal{N}_{g}$ the embedding of $\Phi_{g}$ can be homotoped to have image in the intersection $\partial\mathcal{T}_{g}^{\delta}\cap \mathrm{Min}(C)$. Composing with $\pi:\mathcal{N}_{g}\rightarrow \mathcal{C}_{g}^{\circ}$ gives a homotopy in $\mathcal{C}_{g}^{\circ}$. In this way an embedding of a $2g-2$-sphere in $\mathcal{C}_{g}^{\circ}$ denoted by $q_{C,g}$ is obtained.

\begin{lemma}
\label{maintheorem}
For $g\geq 2$, $q^{\circ}_{g}$ is $\Gamma_{g}$-equivariantly homotopic to $q_{C,g}$.
\end{lemma}
\begin{proof}

Suppose $C_{i}$ for $i=1, \ldots, n$ are nonfilling subsets of $C$, each of which is obtained by deleting a pair of disjoint curves from $C$, and such that $\{m(C_{1}), \ldots, m(C_{n})\}$ are pairwise disjoint.\\

Claim 1 - Suppose $\sigma$ is a top dimensional simplex in the image of the embedding of $\Phi_{g}$ by $q_{C,g}$ defined above, with vertices given by $\{v(C_{1}),\ldots, v(C_{n})\}$. Then for sufficiently small $\delta$, everywhere on $\sigma$ the length of any multicurve in the set $\{m(C_{1}),\ldots, m(C_{n})\}$ is less than the diameter of the collar of a geodesic of length $\delta$.\\

Note that in the $\delta$-thin part of $\mathcal{T}_{g}^{\delta}$, the statement that a given curve has length less than the diameter of a collar of a geodesic of length $\delta$ amounts to saying that the curve is disjoint from the systoles or is a systole.\\

Proof of Claim 1: Using Keen's collar lemma, the bound on the Margulis constant given in \cite{Yamada} and some hyperbolic trigonometry, for example Theorem 3.5.4 of \cite{Ratcliffe}, it is not hard to check that the length of the systoles at $p$ in $\mathrm{Min}(C)$ is less than the diameter of the collar of a geodesic of length $\delta$. Similarly, for sufficiently small $\delta$, the length of $m(C')$ at $p$ is less than the diameter of the collar of a geodesic of length $\delta$. Moreover, Lemma \ref{Thurstonprop} implies that the length function $L(C')$, for nonfilling $C'\subsetneq C$, has some values on $\partial\mathcal{T}_{g}^{\delta}$ smaller than at $p$.\\

Denote by $\Sigma(C_{i})$ the union of cells each containing the vertex $v(C_{i})$. At the vertex $v(C_{i})$ the systoles are disjoint from every curve in $C_{i}$ and hence from $m(C_{i})$. Otherwise, for sufficiently small $\delta$, one of the systoles would intersect a curve in $C_{i}$, making $L(C_{i})$ larger than its value at $p$. The systoles are therefore contained in the multicurve $m(C_{i})$ at $v(C_{i})$. Any curves in $m(C_{i})$ that are not systoles at the vertex must also have length less than the diameter of the collar of a geodesic of length $\delta$. This is because such curves are homotopic to curves lying along arcs of the curves $C_{i}$ and hence cannot be too long without forcing $L(C_{i})$ to be larger than its value at $p$.\\

At the vertex $v(C_{j})$ of $\Sigma(C_{i})$, each curve in $m(C_{i})$ is either homotopic to a curve in $m(C_{j})$ or to a piecewise smooth curve lying along some of the geodesic representatives of the curves in $C_{j}$. Either way, for sufficiently small $\delta$, the length of $m(C_{j})$ is less than the diameter of the collar of a geodesic of length $\delta$. The same is true for any vertex or midpoint of a simplex in $\Sigma(C_{i})$. This can then be extended to $\Sigma(C_{i})$, because the length functions minimised over the cells are linear combinations of the length functions minimised at the vertices and midpoints of cells. Since $\sigma$ is in the intersection $$\cap_{j=1}^{n}\Sigma(C_{j})$$ the claim follows.\\

Claim 2 - Suppose $\sigma$ is a top dimensional simplex with vertices $\{v(C_{1}),\ldots, v(C_{n})\}$ as in Claim 1. Then for sufficiently small $\delta$, if $c$ is a systole somewhere on $\sigma$, $c$ is contained in one of the multicurves $\{m(C_{1}),\ldots, m(C_{n})\}$.\\

Claim 2 is a corollary of Claim 1. Since $\sigma$ is top dimensional, it follows from the construction that the multicurves $\{m(C_{1}),\ldots, m(C_{n})\}$ determine a pants decomposition of $\mathcal{S}_{g}$. If $c$ were a systole on $\sigma$ intersecting one of the multicurves $\{m(C_{1}),\ldots, m(C_{n})\}$, this would contradict Claim 1.\\


A $\Gamma_{g}$-equivariant homotopy between $q_{C,g}$ and $q^{\circ}_{g}$ will now be constructed. Let $\sigma$ be the image under $q_{C,g}$ of a top dimensional cell of the barycentric subdivision of $\Phi_{g}$. The vertices of $\sigma$ are already on $\mathcal{C}_{g}^{\circ}$. A vertex $v(C')$ will be mapped to the vertex $v(C')^{\mathcal{C}}$ of $\mathcal{C}_{g}^{\circ}$ labelled by $m(C')$. Recall that by Claim 2, the systoles at $v(C')$ are a submulticurve of $m(C')$. So $v(C')$ is already contained in a cell in $\mathcal{C}_{g}$ containing $v(C')^{\mathcal{C}}$. The homotopy shifts $v(C')$ to $v(C')^{\mathcal{C}}$ along a straight line contained within this simplex. \\

Now suppose $e(C', C'')$ is an edge of the embedding of the barycentric subdivision of $\Phi_{g}$. Here $v(C')$ is a vertex and $v(C'')$ the midpoint. Then by Claim 2, the edge $e(C', C'')$ is contained in a thickened simplex containing the edge $e(C', C'')^{\mathcal{C}}$ in the image of $\mathcal{C}_{g}^{\circ}$ joining $v(C')^{\mathcal{C}}$ and $v(C'')^{\mathcal{C}}$. The homotopy maps $e(C', C'')$ to $e(C', C'')^{\mathcal{C}}$ by shifting points along straight lines within this thickened simplex, extending the homotopy taking the endpoints of $e(C', C'')$ to the endpoints of $e(C', C'')^{\mathcal{C}}$.\\

The homotopy can be defined inductively, where the homotopy on the $n+1$-skeleton extends the homotopy of the $n$-skeleton.
\end{proof}

\section{Proof of the Conjecture}
\label{finalproof}

This section gives a proof of the main theorem and its corollary.\\

\begin{theorem}
\label{blahblah}
$\mathrm{St}(S_{g})$ is generated by the $\Gamma_{g}$-orbit of the image in $\mathcal{C}_{g}$ of the map $q_{g}$.
\end{theorem}

\begin{proof}

The theorem will be proven by first proving the lemma below. As it was shown in Theorem 4.2 of \cite{Broaddus} that $\mathrm{St}(\mathcal{S}_{g})$ is cyclic modulo the action of $\Gamma_{g}$, it then sufficies to show primitivity of the homology class with representative given by the image of the map $q_{g}$.

\begin{lemma}[Conjecture 4.9 of \cite{Broaddus}]
\label{Blem}
The image of the map $q_{g}$ represents a nontrivial element of $\tilde{H}_{2g-2}(\mathcal{C}_{g};\mathbb{Z})$.
\end{lemma}

For convenience, the map $q^{\circ}_{g}$ with image in $\mathcal{C}_{g}^{\circ}$ will be used in place of $q_{g}$. This can be done because it was shown that the two maps are homotopic.\\

The cases $g=1$ and $g=2$ of the lemma were proven in \cite{Broaddus}, where the first nontrivial case $g=2$ was worked through explicitly. The sphere obtained in the case $g=2$ is shown in Figure \ref{genus2sphere}.\\

For simplicity, balls and spheres in $\mathcal{N}_{g}\cap \partial\mathcal{T}_{g}^{\delta}$ will be confused with corresponding objects in $\mathcal{C}_{g}$. This is done by identifying the image of the embedding of $\mathcal{C}_{g}$ with $\mathcal{C}_{g}$, and projecting $\mathcal{N}_{g}$ onto this image via the systole-preserving map $\pi$.\\

Let $C$ be the set of curves from Example \ref{Schmutzexample}. It follows from Lemma \ref{maintheorem} and the definition of $q_{C,g}$ that the image of $q^{\circ}_{g}$ represents a boundary in $\mathcal{C}_{g}^{\circ}$ iff the ball $\mathcal{B}_{g}$ in $\mathrm{Min}(C)$ bounded by the image of the embedding of $\Phi_{g}$ can be homotoped relative to its boundary into $\partial\mathcal{T}_{g}^{\delta}$. \\

Suppose there exists a homotopy of $\mathcal{B}_{g}$ relative to its boundary into $\partial \mathcal{T}_{g}^{\delta}$. Denote by $\mathcal{B}_{g}^{\mathcal{C}}$ a ball in $\mathcal{C}_{g}^{\circ}$ homotopic to  $\mathcal{B}_{g}$. It is possible to assume without loss of generality that $\mathcal{B}_{g}^{\mathcal{C}}$ is invariant under the action of the finite subgroup of $\Gamma_{g}$ that stabilises $\partial\mathcal{B}_{g}^{\mathcal{C}}$. If $\mathcal{B}_{g}^\mathcal{C}$ could not be chosen this way, Theorem 4.2 of \cite{Broaddus} would imply that $\mathrm{St}(S_{g})$ is finite modulo the action of the mapping class group.\\

The critical point $p$ from Example \ref{Schmutzexample} has an automorphism group that acts transitively on the curves in $C$. This automorphism group therefore corresponds to a subgroup of $\Gamma_{g}$ that stabilises $\mathrm{Min}(C)$. Since the action of the mapping class group preserves injectivity radius, this subgroup also stabilises $\mathcal{B}_{g}$ and its boundary. This group will be called $G_{C}$ from now on. $G_{C}$ contains a cyclic element of order $2g+2$ that permutes the systoles, preserving the cyclic ordering. In addition, it contains elements that correspond to a reflection of the $(2g+2)$-gon $K_{2g+2}$ from Section \ref{Broaddussec}.\\





A unique path $\gamma$ that is the fixed point set of this subgroup of $G_{C}$ will now be described. This is an example of a type of path constructed in \cite{Sanki}.\\

At the critical point $p$ in Example \ref{Schmutzexample}, the systoles intersect at right angles, and cut the surface into right angled $(2g+2)$-gons. The set $C$ can be decomposed into two multicurves, each with $g+1$ elements; these are the ``A-curves'' and the ``B-curves''. The edges of each $(2g+2)$-gon alternately lie along A-curves and B-curves. Each systole lies along two edges of this tesselation, and these two edges are on the boundary of different $(2g+2)$-gons. \\

It is possible to deform the hyperbolic structure such that the tesselation by right angled $(2g+2)$-gons becomes a tesselation by $(2g+2)$-gons with all edge lengths equal and angles that alternate between $\pi/2+t$ and $\pi/2-t$ for $t\in [0,\pi/2)$. This is done in such a way that all angles at the vertices of the tesselation still sum to $2\pi$ and opposite pairs of edges at vertices continue to meet at angle $\pi$. The hyperbolic structure on the $(2g+2)$-gons extends to a hyperbolic structure on the tesselated surface. A map $\gamma:(-\pi/2, \pi/2)\rightarrow \mathcal{T}_{g}$ is obtained. Here $t\in (-\pi/2,\pi/2)$ is mapped to the point in $\mathcal{T}_{g}$ represented by the hyperbolic surface tesselated by the complements of the geodesics in $C$, where the tesselation consists of $(2g+2)$-gons with angles alternating between $\pi/2+t$ and $\pi/2-t$.\\

The path $\gamma$ is the fixed point set of the subgroup $G_{C}$ of $\Gamma_{g}$. Since $G_{C}$ acts by isometry, it is a necessary condition that the lengths of the curves in $C$ are all equal at every point of the fixed point set, similarly for the angles of intersection, where these are defined appropriately using an orientation convention as in the definition of $\gamma$. The action of $G_{C}$ at $\gamma(t)$ performs the same permutation of the $(2g+2)$-gons of the tesselation of $\gamma(t)$ and their edges as at the critical point $\gamma(0)$. The fixed point sets of the reflections contain $\gamma$, and extend radially outwards from $\gamma$, intersecting $\partial\mathcal{B}_{g}$. The Brouwer fixed point theorem implies that the fixed point set of the cyclic subgroup of $G_{C}$ must intersect $\mathcal{B}_{g}^{C}$, similarly for the fixed point set of the reflections. It follows that $\gamma$ intersects $\mathcal{B}_{g}^{C}$.  \\

The quotient $\mathcal{T}_{g}/G_{C}$ is therefore an orbifold with trivial fundamental group and cone singularities along $\gamma$ and radiating outwards from $\gamma$, as explained above. The homotopy from  $\mathcal{B}_{g}$ to $\mathcal{B}_{g}^{\mathcal{C}}$ can therefore be assumed to be $G_{C}$-equivariant; its projection to $\mathcal{T}_{g}/G_{C}$ is a homotopy between two singular balls with the same boundary. This homotopy shifts points on each fixed point set along the fixed point set.\\

Again due to the fact that $G_{C}$ acts by isometry, $\gamma(t)$ is a geodesic with respect to any $\Gamma_{g}$-equivariant metric on $\mathcal{T}_{g}$. In particular, it is a geodesic with respect to the Teichm\"uller metric. If $\gamma$ is the axis of a pseudo-Anosov (calculations suggest this is not the case), the image of $\gamma$ in the moduli space is a closed loop  in the $\delta$-thick part of $\mathcal{T}_{g}$ for some $\delta$. Since a $\Gamma_{g}$-equivariant homotopy of the ball $\mathcal{B}_{g}$ can only shift $\mathcal{B}_{g}$ in such a way that the image of the critical point stays on $\gamma$, this contradicts the existence of a homotopy.\\

When $\gamma$ is not the axis of a pseudo-Anosov, it intersects $\partial \mathcal{T}_{g}^{\delta}$ in a discrete set of points. This follows from the observation that $\gamma$ is also a Weil-Petersson geodesic, and curve lengths are strictly convex along Weil-Petersson geodesics, \cite{Wolpert}.\\

If $\mathcal{B}_{g}$ is homotopic into $\partial\mathcal{T}_{g}^{\delta}$, by symmetry, there must be at least two $G_{C}$-equivariant homotopies; one moving a point of $\mathcal{B}_{g}$ in one direction along $\gamma$, and the other in the opposite direction along $\gamma$. These homotopies take $\mathcal{B}_{g}$ to distinct balls $\mathcal{B}_{g}^{L}$ and $\mathcal{B}_{g}^{R}$ in $\mathcal{C}_{g}$, as demonstrated by the fact that the set of systoles at the two corresponding points of intersection of $\gamma$ with $\partial \mathcal{T}_{g}^{\delta}$ are different. Both these balls in $\mathcal{C}_{g}$ have dimension $2g-1$, as they have $2g-2$-dimensional boundary.\\

As $\mathcal{C}_{g}$ has the homotopy type of a wedge of spheres of dimension $2g-2$, the $(2g-1)$-sphere in $\mathcal{C}_{g}$ obtained by gluing $\mathcal{B}_{g}^{L}$ and $\mathcal{B}_{g}^{R}$ together along their boundaries must bound a ball in $\mathcal{C}_{g}$. As this is also the case in the projection to $\mathcal{T}_{g}/G_{C}$, the $\Gamma_{g}$ orbits of $\mathcal{B}_{g}^{L}$ and $\mathcal{B}_{g}^{R}$ must be $G_{C}$-equivariantly homotopic in $\mathcal{C}_{g}$. However, this contradicts the fact that $\gamma$ only intersects $\partial \mathcal{T}_{g}^{\delta}$ in a discrete set of points. It follows that there can be no homotopy, proving Lemma \ref{Blem}.\\

It remains to show primitivity of the image of $q_{g}$.  By Lemma \ref{maintheorem}, the image of $q_{g}$ maps to an embedded sphere in the moduli space $\mathcal{M}_{g}$, representing the primitive homology class of the $\delta$-thin part of $\mathcal{M}_{g}$ that is killed by attaching the $(2g-1)$-handle corresponding to the $\Gamma_{g}$-orbit of critical points or sets of minima from Example \ref{Schmutzexample}. If the image of $q_{g}$ does not represent a primitive homology class, it is a multiple of a primitive homology class $[q_{g}^{prim}]$ with stabiliser subgroup $G_{C}^{prim}$ strictly containing $G_{C}$.\\

It follows from the Nielsen Realisation Theorem that every finite subgroup of $\Gamma_{g}$ has a fixed point set. This observation, together with the invariance of injectivity radius and the Weil-Petersson metric under the action of $\Gamma_{g}$, can be used to construct a ball $\mathcal{B}_{g}^{prim}$ in $\mathcal{T}_{g}^{\delta}$, invariant under $G_{C}^{prim}$ and with boundary projecting onto a representative of $[q_{g}^{prim}]$ in the boundary of the $\delta$-thick part of $\mathcal{M}_{g}$. \\

The fixed point set of $G_{C}^{prim}$ must be on $\gamma$. Since $G_{C}$ is the largest subgroup that fixes $\gamma$, the fixed point set of $G_{C}^{prim}$ must be a point $p_{fix}$ on $\gamma$. This is a point at which the images of $\gamma$ under the mapping classes in $G_{C}^{prim}\setminus G_{C}$ intersect. There is a $G_{C}$-equivariant homotopy of $\mathrm{Min}(C)$ along $\gamma$, taking the critical point $p$ to $p_{fix}$, whose restriction to $\partial \mathcal{T}_{g}^{\delta}$ is a $G_{C}$-equivariant homotopy. The image of $\mathrm{Min}(C)$ under this homotopy, call it $\mathrm{Min}(C)^{fix}$, is invariant under the action of $G_{C}$. The balls in the orbit of $\mathrm{Min}(C)^{fix}$ under the action of $G_{C}^{prim}$ intersect along the largest dimensional ball invariant under $G_{C}^{prim}$. However, this ball has dimension strictly less than $\mathrm{Min}(C)$, so its boundary cannot project to a  primitive homology class of the $\delta$-thin part of $\mathcal{M}_{g}$. This concludes the proof of primitivity, and hence of the theorem.

\end{proof}

\begin{corollary}
\label{kernel}
Denote by $Z_{g}$ the center of the mapping class group $\Gamma_{g}$. Then the stabiliser subgroup of the action of $\Gamma_{g}/Z_{g}$ on the homology class $[q_{g}]$ in $\mathrm{St}(S_{g})$ with representative the image of $q_{g}$ is isomorphic to the dihedral group $D_{2g+2}$. Moreover, the kernel of the action of $\Gamma_{g}/Z_{g}$ on $\mathrm{St}(S_{g})$ is trivial.
\end{corollary}
\begin{proof}
Recall the construction of the simplicial $2g-2$ dimensional sphere $\Phi_{g}$ via the $(2g+2)$-gon $K_{2g+2}$. It follows that the automorphism group of the associahedron $\mathcal{K}_{2g-1}$ is isomorphic to the automorphism group of $K_{2g+2}$, namely the dihedral group $D_{2g+2}$. \\

\begin{figure}
\centering
\includegraphics[width=0.6\textwidth]{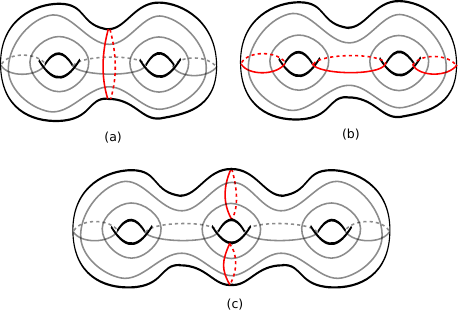}
\caption{In genus 2, the mapping class in the proof of Corollary \ref{kernel} is a composition of reflections through the red curve shown in part (a), followed by the red multicurve in part (b). When the genus is odd, the first reflection is through a multicurve with two connected components, as shown in part (c).}
\label{2reflections}
\end{figure}

The dihedral group $D_{2g+2}$ is generated by reflections. Let $r$ be one such reflection. Since $2g+2$ is even, $r$ leaves invariant a pair of vertices of $K_{2g+2}$; let $c_{1}$ and $c_{2}$ be the curves in $C$ labelling these edges. There is an element $\alpha(r)$ of the mapping class group that realises $r$ in the sense that the induced action of $\alpha(r)$ on the curves of $\mathcal{S}_{g}$ determines the same permutation of the elements of $C$ as the action of $r$ on the vertices of $K_{2g+2}$. The mapping class $\alpha(r)$ is a composition of two reflections of $\mathcal{S}_{g}$, as illustrated in Figure \ref{2reflections}. The first reflection is through a multicurve that intersects $c_{1}$ and $c_{2}$ twice each, and cuts each of the $(2g+2)$-gons in $S_{g}\setminus C$ into two $(g+2)$-gons. This multicurve is a curve when the genus is even, and has 2 connected components when the genus is odd, as illustrated in Figure \ref{2reflections} (a) and (c). The second reflection is through the set of curves obtained as the boundary curves of a subsurface obtained by gluing two adjacent $(2g+2)$-gons in $S_{g}\setminus C$ along their common edges as shown in Figure \ref{2reflections} (b). This concludes the proof of the first part of the Corollary.\\

The kernel of the action of $\Gamma_{g}/Z_{g}$ on $\mathrm{St}(S_{g})$ is contained in the intersection of the stabiliser subgroups of the generators, each of which can be taken to be in the $\Gamma_{g}$ orbit of a fixed $[q_{g}]$ by Theorem \ref{blahblah}. These stabiliser subgroups consist of all the conjugates of the stabiliser subgroup of $[q_{g}]$. This intersection can readily be seen to be trivial; consider for example the conjugate by a large power of a pseudo-Anosov. The intersection with the stabiliser subgroup of $[q_{g}]$ is zero, because the images of the reflections fix different curves. 

\end{proof}

\bibliography{spinebib2}
\bibliographystyle{plain}
\end{document}